\documentclass{amsart}

\usepackage{enumerate}

\numberwithin{equation}{section}

\newcommand{\bC}{\mathbb{C}}

\newcommand{\bE}{\mathbb{E}}

\newcommand{\cA}{\mathcal{A}}

\newcommand{\cK}{\mathcal{K}}
\newcommand{\cL}{\mathcal{L}}

\newcommand{\fS}{\mathfrak{S}}

\newcommand{\Det}{\mathrm{\,Det\,}}

\newcommand{\per}{\mathrm{\,per\,}}

\theoremstyle{plain} \newtheorem{Theo}{Theorem}[section]
\theoremstyle{plain} 
\theoremstyle{plain} 
\theoremstyle{plain} \newtheorem{Prop}[Theo]{Proposition}
\theoremstyle{remark} \newtheorem{Rem}[Theo]{Remark}
\theoremstyle{definition} 
\theoremstyle{definition} 
\theoremstyle{remark}

\begin{document}

\title{Hyperdeterminantal point processes}

\author{Steven N. Evans}

\email{evans@stat.Berkeley.EDU}

\address{Department of Statistics \#3860 \\
 University of California at Berkeley \\
367 Evans Hall \\
Berkeley, CA 94720-3860 \\
U.S.A}

\thanks{SNE supported in part by NSF grant DMS-0405778}

\author{Alex Gottlieb}

\email{alex@alexgottlieb.com}

\address{Wolfgang Pauli Institute \\
c/o Faculty of Mathematics \\
UZA 4 (7th floor, Green Area ``C'') \\
Nordbergstrasse 15 \\
1090 Wien \\
AUSTRIA
}

\thanks{AG supported by the Vienna Science and Technology Fund, via the project ``Correlation in quantum systems".}

\date{\today}



\begin{abstract}  
As well as arising
naturally in the study of
non-intersecting random paths, random spanning trees,
and eigenvalues of random matrices,
determinantal point processes (sometimes also
called fermionic point processes)
are relatively easy to simulate
and provide a quite broad class of models
that exhibit repulsion between points.  The fundamental
ingredient used to construct a determinantal point process
is a kernel giving the pairwise interactions between points:
the joint distribution of any number of points then has
a simple expression in terms of
determinants of certain matrices defined from this kernel.
In this paper we initiate the study of an analogous class of
point processes that are defined in terms of a kernel
giving the interaction between $2M$ points for some integer
$M$.  The role of matrices is now played by $2M$-dimensional
``hypercubic'' arrays, and the determinant is replaced
by a suitable generalization of it to such arrays -- Cayley's
first hyperdeterminant.  We show that some of the desirable
features of determinantal point processes continue to be
exhibited by this generalization.
\end{abstract}

\maketitle

\section{Introduction}

Motivated by considerations of the behavior of 
fermions in quantum mechanics,
determinantal point processes were introduced in \cite{MR0380979}.  Surveys of their properties and numerous applications
may be found in \cite{MR950166, MR1799012, MR2031202, MR2216966,
MR1940779, MR2018415, MR1989442, MR2073340}.

We consider a certain extension of
this class of point processes.  
In order to motivate our generalization,
we first consider a particular case of the determinantal point
process construction.  Suppose that on some measure space 
$(\Sigma, \cA, \mu)$ we have a kernel $K:\Sigma^2 \rightarrow \bC$ 
that defines an $L$
dimensional projection operator for $L^2(\mu)$.  That is,
\begin{itemize}
\item
$K(x;y) = \bar K(y;x)$ for all $x,y \in \Sigma$,
\item
$\sum_{i,j=1}^n K(x_i; x_j) z_i \bar z_j \ge 0$ for all $x_1, \ldots, x_n
\in \Sigma$ and $z_1, \ldots, z_n \in \bC$,
\item
$\int_\Sigma K(x;y) K(y;z) \, \mu(dy) = K(x;z)$ for all $x,z \in \Sigma$,
\item
$\int_\Sigma K(x;x) \, \mu(dx) = L$.
\end{itemize}
The corresponding determinantal point process can then be thought
of as an exchangeable random vector with values in $\Sigma^L$.
The distribution of this random vector is a probability measure
that has the density
\[
(x_1, \ldots, x_L) \mapsto (L !)^{-1} \det (K(x_i; x_j))_{i,j=1}^L
\]
with respect to the measure $\mu^{\otimes L}$. 

One of the most agreeable
things about this construction is that for $1 \le N \le L$
the $N$-dimensional marginal distributions 
of the random vector have (common) density
\[
(x_1, \ldots, x_N) \mapsto (L(L-1)\cdots(L-N+1))^{-1} \det (K(x_i; x_j))_{i,j=1}^N
\]
with respect to the measure $\mu^{\otimes N}$.  Consequently,
the conditional distribution of the 
$(N+1)^{\mathrm{st}}$ component
of the $\Sigma^L$-valued random vector given the first $N$ 
components can be computed explicitly (as a constant
multiple of a ratio of determinants). It is thus
possible to simulate the entire $\Sigma^L$-valued
random vector if one is able to simulate a general
$\Sigma$-valued random variable from a knowledge of its 
probability density function.

Various generalizations of determinantal point processes
have appeared in the literature.  Note that
\[
\det (K(x_i; x_j))_{i,j=1}^L
=
\sum_{\sigma \in {\fS_L}} \epsilon(\sigma)\prod_{k=1}^L K(x_k; x_{\sigma(k)}),
\]
where $\fS_L$ is the symmetric group of permutations of $\{1, \ldots, L\}$
and $\epsilon$ is the usual alternating character on the symmetric group
(that is, the sign of a permutation).
It is  natural to replace $\epsilon$ by other class functions
on the symmetric group (that is, by other functions
that only depend on the cycle structure of a permutation
and hence are constant on conjugacy classes of the
symmetric group).  
The most obvious choice is to replace
$\epsilon$ by the trivial character which always takes the value
$1$, thereby turning the determinant into a permanent.  Permanental
point processes arise in the description of bosons and are discussed
in \cite{MR0380979, MR2216966, MR950166, MR2018415, MR2073340}.
Replacing $\epsilon$ by a general irreducible character gives
the immanantal point processes of \cite{MR1780025}, while setting
$\epsilon(\sigma) = \alpha^{L-\nu(\sigma)}$ for $-1 < \alpha < 1$
and $\nu(\sigma)$ the number of cycles of $\sigma$ gives the alpha-permanental
processes introduced in \cite{MR1450811}
and further studied in \cite{MR2018415, MR2216966}.

All of these constructions have the feature that an exchangeable joint density
is built up as a linear combination of products of pairwise
interactions.  In this paper we investigate the possibility
of building up a tractable joint density as a linear combination
of products of higher order interactions. In order to accomplish
such a generalization, it is necessary to have higher order
counterparts for both projection kernels and determinants.

Note that $K:\Sigma^2 \rightarrow \bC$ is the kernel of an 
$L$-dimensional projection if and only if
\[
K(y;z) = \sum_{\ell=1}^L  \phi_{\ell}(y) \bar \phi_{\ell}(z),
\]
where $\phi_1, \ldots, \phi_L$ are orthonormal in $L^2(\mu)$.
One possible $(2M)^{\mathrm{th}}$ order extension of
this second order definition is to suppose that:
\begin{itemize}
\item
the underlying space $\Sigma$ is a Cartesian product
$\Sigma_1 \times \cdots \times \Sigma_M$, 
\item 
the measure $\mu$ on $\Sigma$
is a product measure $\mu_1 \otimes \cdots \otimes \mu_M$,
\item 
the functions 
$\phi_{m \ell}: \Sigma \rightarrow \bC$, $1 \le m \le M$, 
$1 \le \ell \le L$, are given by 
$\phi_{m \ell}(x_1, \ldots, x_M) = \psi_{m \ell}(x_m)$,
where for $1 \le m \le M$
the functions $\psi_{m \ell}: \Sigma_m \rightarrow \bC$,
$1 \le \ell \le L$,
belong to $L^2(\mu_m)$ and
are orthonormal in $L^2(\mu_m)$,
\item
the kernel $K: \Sigma^{2M} \rightarrow \bC$ is given by
\[
\begin{split}
K(y_1, \ldots, y_M; z_1, \ldots, z_M)
& :=
\sum_{\ell=1}^L \prod_{m=1}^M \phi_{m \ell}(y_m) \bar \phi_{m \ell}(z_m) \\
& =
\sum_{\ell=1}^L \prod_{m=1}^M \psi_{m \ell}(y_{m m}) \bar \psi_{m \ell}(z_{m m}). \\
\end{split}
\]
\end{itemize}

Note that the integral
\[
\begin{split}
& \int_\Sigma 
\left[
\prod_{m=1}^M \phi_{m \ell'(m)}(x) \bar \phi_{m \ell''(m)}(x)
\right]
\, \mu(dx) \\
& \quad =
\prod_{m=1}^M 
\int_{\Sigma_m} 
\psi_{m \ell'(m)}(x_m) \bar \psi_{m \ell''(m)}(x_m)
\, \mu_m(dx_m) \\
\end{split}
\]
is $1$ if $\ell'(m) = \ell''(m)$ for $1 \le m \le M$,
and the integral is $0$ otherwise. This is analogous
to the orthonormality of the
functions $\phi_1, \ldots, \phi_L$ appearing in the representation
above of an $L$-dimensional projection, and when
$M=1$ we just recover that representation.

The appropriate generalization of the determinant is given by  Cayley's first hyperdeterminant that was introduced
in \cite{Cay43} and which we will describe shortly.  
Cayley later introduced other generalizations of the determinant
that he also called hyperdeterminants and are more
natural from the point of view of invariant theory -- see 
\cite{MR1196989, MR1264417}.   Early treatments of the
theory related to Cayley's original definition may be found in
\cite{Pas00, MR0114826, MR1506358, Ric30, MR1523115, MR1503184,
MR1503183, MR1501856, MR0001195}.  More recent works are
\cite{MR0130256, MR0352115}.  We remark that Cayley's first
hyperdeterminant has been useful in matroid theory 
\cite{MR1355700, Gly06} and we also note the interesting papers
\cite{MR1985318, LuqThi04} in which the calculation of Selberg and
Aomoto integrals is reduced to the evaluation of
hyperdeterminants of suitable multi-dimensional arrays.

Suppose that 
\[
\mathbb{A}(i_1, \ldots, i_M; j_1, \ldots, j_M), \quad
1 \le i_1, \ldots, i_M, j_1, \ldots, j_M \le N.  
\]
is a $2M$-way hypercubic matrix (that is, $\mathbb{A}$ is a
a $2M$-dimensional array or tensor that is of the
same length, namely $N$, in each direction).
Suppose further that $\cK$ is a subset of $\{1, \ldots, M\}$.
We define the corresponding
hyperdeterminant of $\mathbb{A}$  to be
\[
\begin{split}
\Det_\cK(\mathbb{A}) & := 
\frac{1}{N!} 
\sum_{\sigma_1 \in \fS_N} \cdots \sum_{\sigma_M \in \fS_N}
\sum_{\tau_1 \in \fS_N} \cdots \sum_{\tau_M \in \fS_N}
\prod_{k \in \cK} \epsilon(\sigma_k) \epsilon(\tau_k) \\
& \quad \times \prod_{n=1}^N
\mathbb{A}(\sigma_1(n), \ldots, \sigma_M(n); \tau_1(n), \ldots \tau_M(n)),\\
\end{split}
\]
where $\fS_N$ is the symmetric group of permutations of $\{1, \ldots, N\}$
and, as above, $\epsilon$ is the  alternating character.  
This definition is just the usual
definition of the hyperdeterminant of a general hypercubic matrix
with a general ``signancy'', except that we have imposed the
restriction that the coordinate directions of the matrix
are grouped in pairs, and each coordinate direction in a pair has the same
signancy. When $M=1$, so that $\mathbb{A}$ is just an $N \times N$ matrix,
$\Det_\cK(\mathbb{A})$ is either the usual determinant or the
permanent, depending on whether 
$\cK$ is $\{1\}$ or $\emptyset$.  

\medskip
{\bf Note: From now on we will assume
that $\cK$ is non-empty.}
\medskip

We are now ready to define a family of exchangeable
probability densities.  For $1 \le N \le L$, define the function
$p_N: \Sigma^N \rightarrow \bC$ by
\[
p_N(x_1, \ldots, x_N) :=  \left(\binom{L}{N} (N!)^M\right)^{-1} \Det_\cK(\mathbb{B}),
\]
where $\mathbb{B}$ is the $2M$-way hypercubic matrix 
of length $N$ given by
\[
\mathbb{B}(i_1, \ldots, i_M; j_1, \ldots, j_M)
:=
K(x_{i_1}, \ldots, x_{i_M}; x_{j_1}, \ldots, x_{j_M}).
\]
That is,
\[
\begin{split}
& p_N(x_1, \ldots, x_N) \\
& \quad =
\left(\binom{L}{N} (N!)^{M+1}\right)^{-1} 
\sum_{\sigma_1 \in \fS_N} \cdots \sum_{\sigma_M \in \fS_N}
\sum_{\tau_1 \in \fS_N} \cdots \sum_{\tau_M \in \fS_N}
\prod_{k \in \cK} \epsilon(\sigma_k) \epsilon(\tau_k) \\
& \qquad \times \prod_{n=1}^N
K(x_{\sigma_1(n)}, \ldots, x_{\sigma_M(n)}; x_{\tau_1(n)}, \ldots , x_{\tau_M(n)}),\\
& \quad =
\left(\binom{L}{N} (N!)^{M+1} \right)^{-1}  
\sum_{\sigma_1 \in \fS_N} \cdots \sum_{\sigma_M \in \fS_N}
\sum_{\tau_1 \in \fS_N} \cdots \sum_{\tau_M \in \fS_N}
\prod_{k \in \cK} \epsilon(\sigma_k) \epsilon(\tau_k) \\
& \qquad \times \prod_{n=1}^N
\sum_{\ell=1}^L \prod_{m=1}^M \psi_{m \ell}(x_{\sigma_m(n) m}) \bar \psi_{m \ell}(x_{\tau_m(n) m}). \\
\end{split}
\]

We will prove the following theorem in Section~\ref{S:proof_of_main}.

\begin{Theo}
\label{T:main}
For $1 \le N \le L$, the function $p_N$ is the density
with respect to $\mu^{\otimes N}$ of an exchangeable probability
measure on $\Sigma^N$. That is,
$p_N \ge 0$, 
\[
\int_{\Sigma^N} p_N(x_1, \ldots, x_N) \, \mu^{\otimes N}(d(x_1, \ldots, x_N)) = 1,
\]
and $p_N$ is a symmetric function of its arguments.  
The probability measure associated with $p_N$ is the
common $N$-dimensional marginal of the probability
measure associated with $p_L$. That is,
\[
p_N(x_1, \ldots, x_N)
=
\int_{\Sigma^{(L-N)}} p_L(x_1, \ldots, x_N, x_{N+1}, \ldots, x_L)
\, \mu^{\otimes (L-N)}(d(x_{N+1}, \ldots, x_L)).
\]
\end{Theo}

The function $p_L$ is the density with respect to $\mu^{\otimes L}$
of a probability measure on 
$\Sigma^L = ( \Sigma_1 \times \cdots \times \Sigma_M)^L \simeq  \Sigma_1^L\times \cdots \times \Sigma_M^L$.
We show in Section~\ref{S:varying_M} that the marginal
of this probability measure on $ \Sigma_1^L\times \cdots \times \Sigma_{M'}^L$
for $1 \le M' < M$ is also given by a hyperdeterminantal construction
(with the kernel $K$ replaced by a suitable function of $2M'$ variables)
as long as $\cK \cap \{1, \ldots, M'\} \ne \emptyset$.

If we regard the exchangeable probability measure on $\Sigma^L$ with
density $p_L$ as the distribution of a point process on $\Sigma$,
then it is natural to inquire about the distribution of the
number of points that fall into a given subset of $\Sigma$.
We find a relatively simple expression for the factorial moments of such
distributions in Section~\ref{S:occupancy}.

The key observation behind many of our arguments
is an expansion of suitable hyperdeterminants that is analogous
to the Cauchy-Binet theorem for ordinary determinants.  This
result is an extension of a lemma from \cite{MR1355700},
and we give the proof in Section~\ref{S:expansion}.

\section{A hyperdeterminant expansion}
\label{S:expansion}

For $M=1$, the following result is a consequence of the Cauchy-Binet
expansion for determinants. (Recall our assumption that $\cK$ is non-empty
and so our hyperdeterminant for $M=1$ is a determinant
rather than a permanent -- the Cauchy-Binet expansion for permanents is somewhat different and involves sums over possibly repeated indices.)
When $M>1$ and $\cK = \{1,2, \ldots, M\}$, the 
result is given by Lemma 3.3 of \cite{MR1355700}.

\begin{Prop}
\label{P:expansion}
Suppose that $\mathbb{A}$ is a $2M$-way hypercubic matrix with
length $N$ in each direction that is of the form
\[
\mathbb{A}(i_1, \ldots, i_M; j_1, \ldots, j_M)
=
\sum_{\ell = 1}^L
A^{(1)}(i_1,\ell) \cdots A^{(M)}(i_M, \ell) \bar A^{(1)}(j_1,\ell) \cdots \bar A^{(M)}(j_M, \ell),
\]
where $A^{(m)}$ is an $N \times L$ matrix and 
$\bar A^{(m)}$ is 
the $N \times L$ matrix obtained by taking the complex conjugates of the
entries of $A^{(m)}$.  Then
$\Det_\cK(\mathbb{A}) = 0$ if $L < N$ and otherwise
\[
\begin{split}
\Det_\cK(\mathbb{A})
& =
\sum_\cL 
\left[\prod_{k \in \cK} \det(A_\cL^{(k)}) \det(\bar A_\cL^{(k)})\right]
\left[\prod_{k \notin \cK} \per(A_\cL^{(k)}) \per(\bar A_\cL^{(k)})\right] \\
& =
\sum_\cL 
\left[\prod_{k \in \cK} |\det(A_\cL^{(k)})|^2 \right]
\left[\prod_{k \notin \cK} |\per(A_\cL^{(k)})|^2 \right], \\
\end{split}
\]
where the sum is over all subsets $\cL$ of $\{1,2, \ldots, L\}$
with cardinality $N$ and $A_\cL^{(k)}$ is the $N \times N$
sub-matrix of the matrix
$A^{(k)}$ formed 
by the columns of $A^{(k)}$ with indices in the set $\cL$.
\end{Prop}

\begin{proof}
We have
\[
\begin{split}
&\Det_\cK(\mathbb{A}) \\ 
& = 
\frac{1}{N!} 
\sum_{\sigma_1 \in \fS_N} \cdots \sum_{\sigma_M \in \fS_N}
\sum_{\tau_1 \in \fS_N} \cdots \sum_{\tau_M \in \fS_N}
\prod_{k \in \cK} \epsilon(\sigma_k) \epsilon(\tau_k) \\
& \quad \times \prod_{n=1}^N
A(\sigma_1(n), \ldots, \sigma_M(n); \tau_1(n), \ldots \tau_M(n)) \\
& =
\frac{1}{N!} 
\sum_{\sigma_1 \in \fS_N} \cdots \sum_{\sigma_M \in \fS_N}
\sum_{\tau_1 \in \fS_N} \cdots \sum_{\tau_M \in \fS_N}
\prod_{k \in \cK} \epsilon(\sigma_k) \epsilon(\tau_k) \\
& \quad \times 
\prod_{n=1}^N
\sum_{\ell = 1}^L
A^{(1)}(\sigma_1(n),\ell) \cdots A^{(M)}(\sigma_M(n), \ell) \bar A^{(1)}(\tau_1(n),\ell) \cdots \bar A^{(M)}(\tau_M(n), \ell) \\
& =
\frac{1}{N!} 
\sum_{\sigma_1 \in \fS_N} \cdots \sum_{\sigma_M \in \fS_N}
\sum_{\tau_1 \in \fS_N} \cdots \sum_{\tau_M \in \fS_N}
\prod_{k \in \cK} \epsilon(\sigma_k) \epsilon(\tau_k) \\
& \quad 
\times 
\sum_{\ell_1 = 1}^L \cdots \sum_{\ell_N = 1}^L 
\prod_{n=1}^N
A^{(1)}(\sigma_1(n), \ell_n) \cdots A^{(M)}(\sigma_M(n), \ell_n) \\
& \qquad \times \bar A^{(1)}(\tau_1(n), \ell_n) \cdots \bar A^{(M)}(\tau_M(n), \ell_n) \\
& =
\frac{1}{N!} 
\sum_{\ell_1 = 1}^L \cdots \sum_{\ell_N = 1}^L 
S(\ell_1,\ell_2,\cdots,\ell_N)
\ ,
\end{split}
\]
where 
\begin{eqnarray*}
 && S(\ell_1,\ell_2,\cdots,\ell_N)
 \ = \ 
 \sum_{\sigma_1 \in \fS_N} \cdots \sum_{\sigma_M \in \fS_N}
\sum_{\tau_1 \in \fS_N} \cdots \sum_{\tau_M \in \fS_N}
\prod_{k \in \cK} \epsilon(\sigma_k) \epsilon(\tau_k) \\
& & \qquad 
\times 
\prod_{n=1}^N
A^{(1)}(\sigma_1(n), \ell_n) \cdots A^{(M)}(\sigma_M(n), \ell_n) \bar A^{(1)}(\tau_1(n), \ell_n) \cdots \bar A^{(M)}(\tau_M(n), \ell_n). \\
\end{eqnarray*}

For $\vec{\ell} = (\ell_1, \ldots, \ell_N) \in \{1,\ldots,L\}^N$
and $m \in \{1,\ldots,M\}$ define an 
$N \times N$ matrix $B_{\vec{\ell}}^{(m)}$ by
\[
B_{\vec{\ell}}^{(m)}(i,j)
:=
A^{(m)}(i,\ell_j).
\]
Then
\[
\begin{split} 
&S\big(\vec{\ell}\ \big) =
\left[
\prod_{k \in \cK}
\left(
\sum_{\sigma \in \fS_N} \epsilon(\sigma) 
\prod_{n=1}^N A^{(k)}(\sigma(n), \ell_n)
\right)
\left(
\sum_{\tau \in \fS_N} \epsilon(\tau) 
\prod_{n=1}^N \bar A^{(k)}(\tau(n), \ell_n)
\right)
\right] \\
& \quad \qquad
\times
\left[
\prod_{k \notin \cK}
\left(
\sum_{\sigma \in \fS_N} 
\prod_{n=1}^N A^{(k)}(\sigma(n), \ell_n)
\right)
\left(
\sum_{\tau \in \fS_N}  
\prod_{n=1}^N \bar A^{(k)}(\tau(n), \ell_n)
\right)
\right] \\
& \qquad =
\left[
\prod_{k \in \cK}
\det(B_{\vec{\ell}}^{(k)})
\det(\bar B_{\vec{\ell}}^{(k)})
\right]
\left[
\prod_{k \notin \cK}
\per(B_{\vec{\ell}}^{(k)})
\per(\bar B_{\vec{\ell}}^{(k)})
\right] \\
& \qquad =
\left[
\prod_{k \in \cK}
|\det(B_{\vec{\ell}}^{(k)})|^2
\right]
\left[
\prod_{k \notin \cK}
|\per(B_{\vec{\ell}}^{(k)})|^2
\right] \\
\end{split}
\]

Note that the rightmost product
is zero unless the entries of the vector
$\vec{\ell}$ are distinct, because in
that case each of the matrices
$B_{\vec{\ell}}^{(k)}$ for $k \in \cK$
will have two equal columns and hence
have zero determinant (recall that
$\cK$ is non-empty).  Moreover,
if $\vec{\ell'} = (\ell_1', \ldots, \ell_N')$
and $\vec{\ell''} = (\ell_1'', \ldots, \ell_N'')$
are two vectors with distinct entries such that
$\{\ell_1', \ldots, \ell_N'\} 
= \{\ell_1'', \ldots, \ell_N''\} = \cL$,
then
\[
|\det(B_{\vec{\ell}}^{(k)})|^2 = |\det(A_\cL^{(k)})|^2
\]
and
\[
|\per(B_{\vec{\ell}}^{(k)})|^2 = |\per(A_\cL^{(k)})|^2
\]
for all $k$,
because permuting the columns of a matrix leaves the permanent
unchanged and either leaves the determinant unchanged or
alters its sign.

The result now follows, because for any
subset $\cL$ of $\{1,2, \ldots, L\}$
with cardinality $N$ there are $N!$ vectors
$\vec{\ell} = (\ell_1, \ldots, \ell_N)$
with $\{\ell_1, \ldots, \ell_N\}  = \cL$.
\end{proof}

\section{Proof of Theorem \ref{T:main}}
\label{S:proof_of_main}

By definition,
\[
   \binom{L}{N} (N!)^M p_N(x_1, \ldots, x_N) = \Det_\cK(\mathbb{B}),
\]
where $\mathbb{B}$ is the $2M$-way hypercubic matrix of length $N$ 
given by
\[
\begin{split}
\mathbb{B}(i_1, \ldots, i_M; j_1, \ldots, j_M)
& =
K(x_{i_1}, \ldots, x_{i_M}; x_{j_1}, \ldots, x_{j_M}) \\
& = 
\sum_{\ell=1}^L \prod_{m=1}^M \phi_{m \ell}(x_{i_m}) \bar \phi_{m \ell}(x_{j_m}) \\
& =
\sum_{\ell = 1}^L
B^{(1)}(i_1,\ell) \cdots B^{(M)}(i_M, \ell) \bar B^{(1)}(j_1,\ell) \cdots \bar B^{(M)}(j_M, \ell), \\
\end{split}
\]
and the $N \times L$ matrix $B^{(m)}$ is given by
\[
B^{(m)}(n,\ell) 
:= 
\phi_{m \ell}(x_n)
\]

By Proposition~\ref{P:expansion},
\[
    \binom{L}{N} (N!)^M p_N(x_1, \ldots, x_N)  
 \quad = \quad
 \sum_\cL 
\left[\prod_{k \in \cK} |\det(B_\cL^{(k)})|^2 \right]
\left[\prod_{k \notin \cK} |\per(B_\cL^{(k)})|^2 \right], 
\]
where the sum is over all subsets $\cL$ of $\{1,2, \ldots, L\}$
with cardinality $N$ and $B_\cL^{(k)}$ is the $N \times N$
sub-matrix of the matrix $B^{(k)}$ formed 
by the columns of $B^{(k)}$ with indices in the set $\cL$.

It follows that
$p_N(x_1, \ldots, x_N) \ge 0$.
Also, since the value of the permanent a matrix is unchanged
by a permutation of the rows and the value
of a determinant is either unchanged or merely
changes sign, the function
$p_N$ is unchanged by a permutation of its arguments.

We have
\[
\begin{split}
&   \binom{L}{N} (N!)^M  p_N(x_1, \ldots, x_N) \\
& \quad =
\sum_\cL 
\left[\prod_{k \in \cK}
\sum_{\sigma_k} \sum_{\tau_k} \epsilon(\sigma_k) \epsilon(\tau_k)
\prod_{n=1}^N \phi_{k \sigma_k(n)}(x_n) \bar  \phi_{k \tau_k(n)}(x_n)
\right] \\
& \qquad \times \left[\prod_{k \notin \cK}
\sum_{\sigma_k} \sum_{\tau_k}
\prod_{n=1}^N \phi_{k \sigma_k(n)}(x_n) \bar  \phi_{k \tau_k(n)}(x_n)
\right] \\
& \quad =
\sum_\cL
\sum_{\sigma_1} \cdots \sum_{\sigma_M}
\sum_{\tau_1} \cdots \sum_{\tau_M}
\prod_{k \in \cK} \epsilon(\sigma_k) \epsilon(\tau_k)
\prod_{m=1}^M
\prod_{n=1}^N \phi_{m \sigma_m(n)}(x_n) \bar \phi_{m \tau_m(n)}(x_n), \\
\end{split}
\]
where $\sigma_k$ and $\tau_k$ in the summations range over
bijective maps from $\{1, \ldots, N\}$ to $\cL$ and 
$\epsilon$ is interpreted in the usual way for such a bijection.

Now the integral
\[
\begin{split}
& \int_\Sigma
\left[ 
\prod_{m=1}^M
\phi_{m \sigma_m(n)}(x_n) \bar \phi_{m \tau_m(n)}(x_n)
\right]
\, \mu(dx_n) \\
& \quad =
\prod_{m=1}^M
\int_{\Sigma_m}
\psi_{m \sigma_m(n)}(x_{nm}) \bar \psi_{m \tau_m(n)}(x_{nm})
\, \mu(dx_{nm}) \\
\end{split}
\]
is equal to $1$ if and only if 
$\sigma_m(n) = \tau_m(n)$ for $1 \le m \le M$, and otherwise
the integral is $0$.  Hence
\[
\int_{\Sigma^N} p_N(x_1, \ldots, x_N) \, \mu^{\otimes N}
(d(x_1, \ldots, x_N))
= 1
\]
and
\[
\int_\Sigma p_N(x_1, \ldots, x_N) \, \mu(dx_N)
=
p_{N-1}(x_1, \ldots, x_{N-1}),
\]
as required.

\section{Varying the order $M$}
\label{S:varying_M}

Beginning with a suitable kernel $K: \Sigma^{2M }= ( \Sigma_1 \times \cdots \times  \Sigma_M)^{2M} \rightarrow \bC$,
we have a built a family of functions $p_N$, $1 \le N \le L$,
where $p_N$ is a probability density on 
$(\Sigma_1 \times \cdots \times  \Sigma_M)^N \simeq  \Sigma_1^N \times \cdots \times \Sigma_M^N$ 
with respect to the measure
$(\bigotimes_{m=1}^M \mu_m)^{\otimes N} \simeq \bigotimes_{m=1}^M \mu_m^{\otimes N}$.  For $1 \le M' < M$, it is natural to ask about the
push-forward of the probability measure corresponding to the density
$p_N$ by the projection map
from $\bigotimes_{m=1}^M \mu_m^{\otimes N}$ to
$\bigotimes_{m=1}^{M'} \mu_m^{\otimes N}$ given by
\[
(x_{nm})_{1 \le n \le N, \, 1 \le m \le M}
\mapsto
(x_{nm})_{1 \le n \le N, \, 1 \le m \le M'}.
\]
The answer is given by repeated applications of the following result.

\begin{Theo}
Suppose that either $M \notin \cK$ or $M \in \cK$ and
$\cK \setminus \{M\} \ne \emptyset$.  
Set $\hat \cK := \cK \setminus \{M\}$,
$\hat \Sigma := \prod_{i=1}^{M-1} \Sigma_m$,
and $\hat \mu = \bigotimes_{i=1}^{M-1} \mu_m$.  
Define a kernel $\hat K: \hat \Sigma^{2(M-1)} \rightarrow \bC$
by 
\[
\hat K(y_1, \ldots, y_{M-1}; z_1, \ldots, z_{M-1})
:=
\sum_{\ell=1}^L \prod_{m=1}^{M-1} 
\phi_{m \ell}(y_m) \bar \phi_{m \ell}(z_m).
\]
The function
\[
\begin{split}
& (x_{nm})_{1 \le n \le N, \, 1 \le m \le M-1}
\mapsto
\hat p_N((x_{nm})_{1 \le n \le N, \, 1 \le m \le M-1}) \\
& \quad :=
\int_{\Sigma_M^N} 
p_N((x_{nm})_{1 \le n \le N, \, 1 \le m \le M})
\, \mu_M^{\otimes N}(d(x_{1M}, \ldots ,x_{NM})) \\
\end{split}
\]
is a probability density with respect to the
measure $\hat \mu^{\otimes N}$.
The probability density $\hat p_N$ is constructed from the
kernel $\hat K$ and the set of indices $\hat \cK$
in the same manner that the probability
density $p_N$ is constructed from the
kernel $K$ and the set of indices $\cK$.
\end{Theo}

\begin{proof}
As in the proof of Theorem \ref{T:main},
\[
\begin{split}
& \binom{L}{N} (N!)^M \
p_N((x_{nm})_{1 \le n \le N, \, 1 \le m \le M}) \\
& \quad =
\sum_\cL
\sum_{\sigma_1} \cdots \sum_{\sigma_M}
\sum_{\tau_1} \cdots \sum_{\tau_M}
\prod_{k \in \cK} \epsilon(\sigma_k) \epsilon(\tau_k)
\prod_{m=1}^M
\prod_{n=1}^N \psi_{m \sigma_m(n)}(x_{nm}) \bar \psi_{m \tau_m(n)}(x_{nm}), \\
\end{split}
\]
where $\sigma_k$ and $\tau_k$ in the summations range over
bijective maps from $\{1, \ldots, N\}$ to $\cL$.

Observe that the integral
\[
\int_{\Sigma_M^N}
\left[
\prod_{n=1}^N 
\psi_{M \sigma_M(n)}(x_{nM}) \bar \psi_{M \tau_M(n)}(x_{nM})
\right]
\, \mu_M^{\otimes N}(d(x_{1M}, \ldots ,x_{NM}))
\]
is $1$ if and only if $\sigma_M(n) = \tau_M(n)$ for 
$1 \le n \le N$ (that is, if and only if 
$\sigma_M = \tau_M$), and the integral is $0$ otherwise.
For each choice of the set $\cL$, there are $N!$ choices
of the pair of bijections $(\sigma_M, \tau_M)$ such that
$\sigma_M = \tau_M$, and for all of these choices we
have, of course, that 
$\epsilon(\sigma_M) = \epsilon(\tau_M)$ if $M \in \cK$.

It follows that 
\[
\begin{split}
& \binom{L}{N} (N!)^M
\int_{\Sigma_M^N} 
p_N((x_{nm})_{1 \le n \le N, \, 1 \le m \le M})
\, \mu_M^{\otimes N}(d(x_{1M}, \ldots ,x_{NM})) \\
& \quad =
N!
\sum_\cL
\sum_{\sigma_1} \cdots \sum_{\sigma_{M-1}}
\sum_{\tau_1} \cdots \sum_{\tau_{M-1}}
\prod_{k \in \hat \cK} \epsilon(\sigma_k) \epsilon(\tau_k) \\
& \qquad \times \prod_{m=1}^{M-1}
\prod_{n=1}^N \psi_{m \sigma_m(n)}(x_{nm}) \bar \psi_{m \tau_m(n)}(x_{nm}), \\
\end{split}
\]
as claimed.
\end{proof}

\section{The number of points falling in a set}
\label{S:occupancy}

Write $(X_1, \ldots, X_L)$ for a $\Sigma^L$-valued
random variable that has the distribution possessing 
density $p_L$ with respect to the measure $\mu^{\otimes L}$.
Given a set $C \in \cA$, let $J_C := \#\{1 \le \ell \le L: X_\ell \in C\}$, so that $J_C$ is a random variable
taking values in the set $\{0,1,\ldots,L\}$.
The distribution of the random variable $J_C$ is determined by
the factorial moments $\bE[J_C(J_C-1) \cdots (J_C-N+1)]$,
$1 \le N \le L$.

If we write $I_\ell$ for the indicator random variable of the
event $\{X_\ell \in C\}$, then $J_C = I_1 + \cdots + I_L$.
By the exchangeability of $(X_1, \ldots, X_L)$, we have
\[
\begin{split}
& \bE[J_C(J_C-1) \cdots (J_C-N+1)] \\
& \quad =
L(L-1)\cdots(L-N+1) \, \bE[I_1 \cdots I_N] \\
& \quad =
L(L-1)\cdots(L-N+1) 
\int_{C^N} p_N(x_1, \ldots, x_N) \, \mu^{\otimes N}(d(x_1, \ldots, x_N)).\\
\end{split}
\]

Suppose now that $C = C_1 \times \cdots \times C_M$, with $C_m \subseteq \Sigma_m$,
$1 \le m \le M$.  Then
\[
\begin{split}
& \bE[J_C(J_C-1) \cdots (J_C-N+1)] \\
& \quad =
(N!)^{-M+1}
\sum_\cL
\sum_{\sigma_1} \cdots \sum_{\sigma_M}
\sum_{\tau_1} \cdots \sum_{\tau_M}
\prod_{k \in \cK} \epsilon(\sigma_k) \epsilon(\tau_k)
\prod_{m=1}^M
\prod_{n=1}^N H^{(m)}_{\sigma_m(n), \tau_m(n)}, \\
\end{split}
\]
where $H^{(m)}$ is the $L \times L$ matrix defined by
\[
H^{(m)}(\ell,\ell') 
:= 
\int_{C_m} \psi_{m \ell}(y) \bar \psi_{m \ell'}(y) \, \mu_m(dy), \quad
1 \le \ell, \ell' \le L.
\]
Thus, 
\[
\bE[J_C(J_C-1) \cdots (J_C-N+1)]
=
N! \sum_\cL
\left[\prod_{k \in \cK} \det(H_{\cL}^{(k)})\right]
\left[\prod_{k \notin \cK} \per(H_{\cL}^{(k)})\right],
\]
where $H_{\cL}^{(k)}$ is the $N \times N$ sub-matrix of
$H^{(k)}$ with rows and columns indexed by $\cL$.

\begin{Rem} When $M=1$ (so we are dealing with a determinantal
point process), the last expression is
just the trace of the $\binom{L}{N} \times \binom{L}{N}$
compound matrix consisting of the minors of $H^{(1)}$
with $N$ rows and columns.  This compound matrix has as
its eigenvalues all the products of the eigenvalues
of $H^{(1)}$ taken $N$ at a time, and so its trace
is the sum of all such products.  This shows that
$J_C$ is distributed as the sum of $L$ independent
Bernoulli random variables that have the eigenvalues
of $H^{(1)}$ as their respective success probabilities
-- a result that appears in \cite{MR1989442, MR2216966}.
We have been unable to find an analogous probabilistic
representation for $J_C$ for general $M$.
\end{Rem}

\def\cprime{$'$} \def\cprime{$'$}
\providecommand{\bysame}{\leavevmode\hbox to3em{\hrulefill}\thinspace}
\providecommand{\MR}{\relax\ifhmode\unskip\space\fi MR }
\providecommand{\MRhref}[2]{%
  \href{http://www.ams.org/mathscinet-getitem?mr=#1}{#2}
}
\providecommand{\href}[2]{#2}

\end{document}